\documentclass{amsart}
\usepackage{amssymb}
\usepackage{amsfonts}
\usepackage{amssymb}
\usepackage{amsmath, accents}
\usepackage{amsthm}
\usepackage{enumerate}
\usepackage{tabularx}
\usepackage{centernot}
\usepackage{mathtools}
\usepackage{stmaryrd}
\usepackage{amsthm,amssymb}
\usepackage{etoolbox}
\usepackage{tikz}
\usepackage{amssymb}
\usetikzlibrary{matrix}
\usepackage{tikz-cd}
\usepackage{tikz}
\usepackage{marginnote}
\definecolor{mygray}{gray}{0.85}
\usepackage[backgroundcolor=mygray,colorinlistoftodos,prependcaption,textsize=small]{todonotes}
\usepackage{xcolor}

\renewcommand{\leq}{\leqslant}
\renewcommand{\geq}{\geqslant}
\renewcommand{\nleq}{\nleqslant}

\makeatletter
\def\subsection{\@startsection{subsection}{3}%
  \z@{.5\linespacing\@plus.7\linespacing}{.3\linespacing}%
  {\bfseries\centering}}
\makeatother

\makeatletter
\def\subsubsection{\@startsection{subsubsection}{3}%
  \z@{.5\linespacing\@plus.7\linespacing}{.3\linespacing}%
  {\centering}}
\makeatother

\makeatletter
\def\myfnt{\ifx\protect\@typeset@protect\expandafter\footnote\else\expandafter\@gobble\fi}
\makeatother

\newtheorem{theorem}{Theorem}[section]

\newtheorem{corollary}[theorem]{Corollary}
\newtheorem{definition}[theorem]{Definition}
\newtheorem{lemma}[theorem]{Lemma}
\newtheorem{proposition}[theorem]{Proposition}
\newtheorem{example}[theorem]{Example}
\newtheorem{problem}[theorem]{Problem}
\newtheorem{question}[theorem]{Question}

\newtheorem{fact}[theorem]{Fact}

\newtheorem{notation}[theorem]{Notation}

\newtheorem{conjecture}[theorem]{Conjecture}

\usepackage{eso-pic}
\usepackage{fancyhdr}
\usepackage{blindtext} % just for the example
\newcommand{\mrm}[1]{\mathrm{#1}}

\usepackage{pdfpages}

\newcommand{\mbf}{\mathbf}

\newcommand{\mbb}{\mathbb}
\newcommand{\op}{\operatorname}
\newcommand{\join}{\vee}
\renewcommand{\Join}{\bigvee}
\newcommand{\meet}{\wedge}

\begin{document}
%%%%%%%%%%%%%%%%%%

\begin{abstract} We start a systematic analysis of the first-order model theory of free lattices. 
%%Our main result is that the $\forall$-theory of any free lattice is decidable.  We also prove several other results. 
Firstly, we prove that the free lattices of finite rank are not positively indistinguishable, as there is a positive $\exists \forall$-sentence true in $\mbf F_3$ and false in $\mbf F_4$. Secondly, we show that every model of $\mrm{Th}(\mbf F_n)$ admits a canonical homomorphism into the profinite-bounded completion $\mbf H_n$ of $\mbf F_n$. Thirdly, we show that $\mbf H_n$ is isomorphic to the Dedekind-MacNeille completion of $\mbf F_n$, and that $\mbf H_n$ is not positively elementarily equivalent to $\mbf F_n$, as there is a positive $\forall\exists$-sentence true in $\mbf H_n$ and false in $\mbf F_n$. Finally, we show that $\mrm{DM}(\mbf F_n)$ is a retract of $\mrm{Id}(\mbf F_n)$ and that for any lattice $\mbf K$ which satisfies Whitman's condition $\mathrm{(W)}$ and which is generated by join prime elements, the three lattices $\mbf K$, $\mrm{DM}(\mbf K)$, and $\mrm{Id}(\mbf K)$ all share the same positive universal first-order theory.
\end{abstract}

%\title{On the Elementary Theory of Free Lattices}
\title{Elementary Properties of Free Lattices}

%\todog{Notes in gray are the standard (usually not major) things.}
%\todor{Notes in red are mathematical points which I do not understand. It does NOT mean that I believe that there are false claims, but these are more substantial points.}
%\todoy{Notes in yellow are somewhat general points (suggestions/questions/comments).}

\thanks{The second author was supported by project PRIN 2022 ``Models, sets and classifications", prot. 2022TECZJA. The second author also wishes to thank the group GNSAGA of the ``Istituto Nazionale di Alta Matematica ``Francesco Severi"" (INDAM) to which he belongs}

\author{J.B. Nation and Gianluca Paolini}

\address{Department of Mathematics, University of Hawaii, Honolulu, HI 96822.}
\email{jb@math.hawaii.edu}

\address{Department of Mathematics ``Giuseppe Peano'', University of Torino, Via Carlo Alberto 10, 10123, Italy.}
\email{gianluca.paolini@unito.it}

\date{\today}
\maketitle

%\tableofcontents

\section{Introduction}
	
	The model theoretic analysis of free objects (in the sense of universal algebra) has a long tradition, dating back to the 1940's with the work of Tarski and his school. Problems in this areas are often pretty hard and they require advanced technology to be solved.  A canonical example of this phenomenon is the solution of Tarski's problem on the elementary equivalence of free groups, which was solved in 2006 independently by Sela~\cite{Sela06} and Kharlampovich \& Myasnikov~\cite{KM06}: %%jt 
 all non-abelian finitely generated free groups are elementarily equivalent, regardless of the number of generators.
 As for any naturally arising mathematical structure (or class of structures), many are the model-theoretic questions that can be asked about it.  In the context of free objects we might argue that the focus has been on the following four fundamental problems:
	\begin{enumerate}[(A)]
	\item (positive) elementary equivalence of free objects of different rank;
	\item characterization of the finitely generated models elementarily equivalent to a given free object of finite rank;
	\item decidability of (fragments of) the (positive) first-order theory of a free object;
	\item analysis of the stability (in the sense of Shelah \cite{shelah}) properties of a free object.
\end{enumerate}

\smallskip
	
	The model theoretic literature is full of such results, where, once again, probably the most advanced results are on the model theory of free groups. Another important case worth mentioning is the one of free abelian groups.  In this case it follows easily from \cite{smez} that free abelian groups of different rank are not elementarily equivalent and that free abelian groups of finite rank are superstable. Another important piece of literature is on the model theory of free algebras in the context of infinitary logic, cf.~in particular the fundamental work of Eklof, Mekler, and Shelah \cite{mekler1, mekler2}.
 
 In universal algebra, among the most natural classes of algebraic structures that occur in nature there are certainly {\em lattices}, and so, as for any variety of algebras, there are {\em free lattices}. In the last 30 years or so, the algebraic study of free lattices has reached a very mature state, as witnessed by the canonical reference \cite{the_book} on this topic (the ``blue book''). Despite the widespread interest of model theorists in free objects and despite the advanced development of the theory of free lattices, at the best of our knowledge, very little is known on the model theory of free lattices. We consider this a sad state of affairs and we think of \mbox{this paper as a starting point for a remedy, hoping that it will sparkle interest.}
	
\medskip
	
	To start our model theoretic analysis we test what is the situation against Problems (A)-(D) above in the context of free lattices. With regard to (D), as it is easy to see, free lattices fail the stability property (once again, in the sense of Shelah \cite{shelah}) very badly, and so, although more refined questions can still be meaningfully asked, this might not be the right starting point. Concerning (A), it is easy to see that free lattices of different finite rank can be distinguished by a $\exists\forall$-sentence of first-order logic, because the generating set is unique. This leaves us then with the following three questions, where we denote by $\mathbf{F}_n$ the free lattice of rank $n$.
\begin{enumerate}[(A)]
	\item Are $\mathbf{F}_n$ and $\mathbf{F}_m$ elementarily equivalent in {\em positive} first-order logic?
	\item Which are the finitely generated lattices elementarily equivalent to $\mathbf{F}_n$?
	\item Is the first-order theory of\/ $\mathbf{F}_n$ decidable?
\end{enumerate}

	Unfortunately, Question (C) resisted our tries, but we hope that this paper will spark some interest in this fundamental question. Notice that in his celebrated paper \cite{whitman}, Whitman solved the word problem for free lattices, exhibiting a natural algorithmic procedure to decide whether two lattice terms are equivalent (modulo the theory of lattices)\footnote{As an historical note: it turns out that Skolem in a famous 1920 paper \cite{ThS} solved the word problem not only for free lattices but for any finitely presented lattice; see \cite{FN2016} for Skolem's solution. However, this section of Skolem's paper remained unknown until it was found by Stanley Burris late in the 20th century.}. In logical terms this means that the positive universal theory of a free lattice is decidable.
\mbox{Thus, as a starting point toward Question (C), we ask:}
	
		\begin{problem}
	 Let $3 \leq n \leq \omega$. Is the $\forall$-theory of\/ $\mbf F_n$ decidable?
\end{problem}

%%	\begin{theorem}[Main Theorem] The $\forall$-theory of any free lattice is decidable.
	%The $\exists$-theory of $\mbf F_\omega$ is decidable. In fact, more strongly, it is decidable whether $\exists x_1, ..., x_n \psi(x_1, ..., x_n)$ ($\psi$ quantifier-free) holds in $\mbf F_m$, for all $m \geq n$. 
%%\end{theorem}

	We then move to Questions (A) and (B), in this respect the situation is more favorable. In particular, concerning (A) we were able to show the following:

	\begin{theorem}\label{positive_theorem} The free lattices $\mbf F_n$ (for $3 \leq n < \omega$) are not positively indistinguishable. In fact there is a $\exists \forall$-positive sentence true in $\mbf F_3$ and false in $\mbf F_4$.
\end{theorem}

	Interestingly, our proof does not extend to $n \geq 4$. We naturally wonder if this is a limitation of our methods or if there is an intrinsic reason for this. Also, we want to mention that Theorem~\ref{positive_theorem} was motivated by the analysis of the positive first-order theory of free semigroups and free monoids from references \cite{positive_1, positive_2}.
	
	\medskip
	
	Finally, we move to Question (B). This is the venue that inspired the most interesting results of this paper, with applications also to infinitely generated models of $\mrm{Th}(\mbf F_n)$. The crucial result in this direction is the following ``Profinite Theorem''.

	\begin{theorem}\label{profinite_theorem} Let $3 \leq n < \omega$ and $\mbf K \equiv \mbf F_n$.
 %$\mbf K \models \mrm{Th}(\mbf F_n)$. 
 Then $\mbf K$ admits a canonical homomorphism $h_{\mathbf{K}}$ into the profinite-bounded  completion $\mbf H_n$ of\/ $\mbf F_n$ (cf. Fact~\ref{bounded_fact}).
%such a homomorphism $h_{\mathbf{K}}$ is onto if $\mbf K$ is $\aleph_1$-saturated. 
Furthermore, $\mbf H_n$ is isomorphic to the Dedekind-MacNeille completion of\/ $\mbf F_n$ and $\mbf H_n \not \equiv \mbf F_n$, in fact there is a positive $\forall\exists$-sentence true in $\mbf H_n$ and false in $\mbf F_n$.
\end{theorem}

    Theorem~\ref{profinite_theorem} led us to investigations related to an old question of Gr\"atzer, that is, which first-order conditions are preserved in passing from $\mbf K$ to $\mrm{Id}(\mbf K)$? An old result of Baker \& Hales \cite{baker} says that $\mbf K$ and $\mrm{Id}(\mbf K)$ share the same positive universal theory.  On the other hand, as observed by Funayama in 1944 \cite{funa}, there are distributive lattices $\mbf K$ such that $\mrm{DM}(\mbf K)$ is not distributive, and so in the case of $\mrm{DM}(\mbf K)$ this preservation of the positive universal theory of $\mbf K$ is not at all to be taken for granted. In this direction, in our next and final theorem we isolate two properties (satisfied by free lattices) of an arbitrary lattice $\mbf K$ which ensure that the three lattices $\mbf K$, $\mrm{DM}(\mbf K)$, and $\mrm{Id}(\mbf K)$ all share the same positive universal theory.

Whitman's solution to the word problem for free lattices uses the following condition, which holds in free lattices:
\[ \mathrm{(W)} \quad s \meet t \leq u \join v \text{ implies } s \leq u \join v \text{ or } t \leq u \join v \text{ or }s \meet t \leq u \text{ or } s \meet t \leq v .    \]

\begin{theorem}\label{the_last_theorem} Let $\mbf K$ be a lattice satisfying Whitman's condition $\mathrm{(W)}$ and which is generated by join prime elements.  Then the three lattices $\mbf K$, $\mrm{DM}(\mbf K)$, and $\mrm{Id}(\mbf K)$ all share the same positive universal theory. Furthermore, in the case $\mbf K = \mbf F_n$,
then $\mrm{DM}(\mbf K)$ is a retract of\/ $\mrm{Id}(\mbf K)$ and so, in particular, the first-order positive theory of\/ $\mrm{Id}(\mbf F_n)$ is contained in the first-order positive theory of\/ $\mrm{DM}(\mbf F_n)$.
\end{theorem}

    Motivated by Theorem~\ref{the_last_theorem}, in Corollary~\ref{cor:props} we show that $\mrm{DM}(\mbf F_n)$ and $\mrm{Id}(\mbf F_n)$ are {\em not} elementarily equivalent. But we leave open the following question.

    \begin{question} Are $\mrm{DM}(\mbf F_n)$ and $\mrm{Id}(\mbf F_n)$ positively elementarily equivalent?
\end{question}

	What we find particularly interesting about Theorem~\ref{profinite_theorem} is that this theorem is reminiscent of the model theory of free abelian groups. In fact in that case the same thing happens, with the crucial difference, though, that $\mathbb{Z}^n$ {\em is} elementarily equivalent to its profinite completion. We notice that the lattice $\mbf H_n$ plays a crucial role also in the lattice theoretic literature, in particular in connection with Day's Theorem, see \cite[Section~2.7]{the_book}. The realization that $\mbf H_n$ is isomorphic to the Dedekind-MacNeille completion of $\mbf F_n$ was the crucial ingredient in showing that $\mbf H_n$ and $\mbf F_n$ are not elementarily equivalent, and in fact this model theoretic question inspired this result, but we believe that this fact is of independent interest and could be further explored by lattice theorists. Furthermore, despite the hopelessness of a classification of the models of $\mrm{Th}(\mbf F_n)$ (recall the instability mentioned above), Theorem~\ref{profinite_theorem} reduces the understanding of models of $\mrm{Th}(\mbf F_n)$ to understanding $\mbf H_n$ and to understanding the equivalence classes induced by $\mrm{ker}({h_{\mathbf{K}}})$.  In fact the source of instability present in $\mbf F_n$ reduces to the fact that such equivalence classes can be in general very complicated. On the other hand, under further assumptions on models of $\mrm{Th}(\mbf F_n)$ there is hope to prove some positive results. For example, in light of Theorem~\ref{profinite_theorem} we might say that a lattice $\mbf K \models \mrm{Th}(\mbf F_n)$ is {\em standard} if the map $h_{\mathbf{K}}$ has range in $\mbf F_n$, that is each element of $\mbf K$ is congruent modulo $\mrm{ker}({h_{\mathbf{K}}})$ to an element of $\mbf F_n$ (notice that the equivalence relation $\mrm{ker}({h_{\mathbf{K}}})$ can be described explicitly, cf. Section~4). In this direction we propose the following conjecture:

%	\begin{theorem}\label{profinite_theorem} Let $3 \leq n < \omega$. Then:
%	\begin{enumerate}[(a)]
%	\item if $\mbf K \models \mrm{Th}(\mbf F_n)$, then it admits a canonical homomorphism $h_{\mathbf{K}}$ into the profinite completion $\mbf H_n$ of~$\mbf F_n$, such a homomorphism $h_{\mathbf{K}}$ is onto if $\mbf K$ is $\aleph_1$-saturated;
%	\item modulo a finite set of elements from $\mbf F_n \leq \mbf H_n$ we have no control on the isomorphism type of the lattices $h^{-1}_{\mathbf{K}}(a)$, for $a \in \mbf H_n$ and $\mbf K \models \mrm{Th}(\mbf F_n)$;
%	\item $\mbf H_n$ is isomorphic to the Dedekind-MacNeille completion of $\mbf F_n$ and $\mbf H_n \not \equiv \mbf F_n$.
%	\end{enumerate}
%\end{theorem}

	\begin{conjecture}\label{finite_gen_models} There is no finitely generated standard \mbox{elementary extension of\/ $\mbf F_n$.}
\end{conjecture}

	We actually further believe that $\mbf F_n$ is the {\em only} finitely generated model of its theory, this property is known in the model theoretic community as first-order rigidity, a property which received quite some attention in recent years, see e.g. the result of Avni-Lubotzky-Meiri showing that irreducible non-uniform higher-rank characteristic zero arithmetic lattices (e.g. $\mrm{SL}_n(\mathbb Z)$ for $n\geq 3$) \mbox{are first-order rigid \cite{ALM19}.}
 Note that ``lattices" in topological groups are not the same as ``lattices"
in our sense of ordered sets with meet and join.

\section{Notation}

	We write lattices in boldface latters, so $\mbf L, \mbf K$, etc. Given a lattice $\mbf L$, we write the sup and inf of $\mbf L$ as $\vee$ and $\wedge$, but when convenient we switch to the ``field notation'', so  $+$ and $\cdot$ for sup and inf, respectively. We write tuples of elements (or variables) as $\mathbf{x} = (x_1, ..., x_n)$. Given a cardinal number $\kappa$ we denote by $\mbf F_\kappa$ the free lattice on $\kappa$-many generators. By the language of lattice theory we mean the language $L = \{\vee, \wedge\}$.  In particular we do not require $0$ and $1$ to be in the language (as we also consider $\mbf F_\kappa$ for infinite $\kappa$ and such lattices do not have max or min).

An element $a$ of a lattice $\mbf L$ is \emph{join irreducible} if
it is not a proper join, i.e., there do not exist $b < a$ and $c < a$ such
that $a = b \join c$.  Equivalently, $a$ is join irreducible if it is not the
join of a finite nonempty set of elements strictly below $a$.  In any lattice
satisfying $\mathrm{(W)}$, no element can be a proper join and a proper meet.
Thus in a free lattice every element is either join irreducible or meet 
irreducible; generators are both.  

On the other hand, in any lattice, the least upper bound of 
$\{ b \in L : b < a \}$ is either $a$ or the unique largest element
$a_*$ below $a$.  In the latter case, we say that $a$ is
\emph{completely join irreducible}.  Denote the set of completely
join irreducible elements of $\mbf L$ by $\op{CJI}(\mbf L)$.
A join irreducible element in a free lattice need not be completely join irreducible:
some are, and some are not.  

The terms \emph{meet irreducible} and \emph{completely meet irreducible}
are defined dually, along with the notation $\op{CMI}(\mbf L)$.

%%\section{Decidability of the $\mrm{\forall}$-theory of $\mbf F_n$}

\section{The positive theory of free lattices}

Let $\op{PTh}(\mbf L)$ denote the positive first-order theory of $\mbf L$.
Since $\mbf F_n$ is a homomorphic image of $\mbf F_{n+1}$, 
we have $\op{PTh}(\mbf F_n)  \supseteq \op{PTh}(\mbf F_{n+1})$. To distinguish them, we seek a positive sentence $\pi(\mbf{x})$ that holds in $\mbf F_n$ but not in $\mbf F_{n+1}$. The aim of this section is to show that we can do this for $n=3$.  However, for $n \geq 4$ it remains open whether there is a positive sentence 
that holds in $\mbf F_n$ but not in $\mbf F_{n+1}$.

\medskip \noindent
Consider the following positive first-order formulas in the language of lattice theory:
\[  \op{NI}(x_1, \dots, x_m): \quad (\op{OR}_{1 \leq i \leq m}  x_i \leq \sum_{j \ne i} x_j) \quad \op{OR} \quad 
             ( \op{OR}_{1 \leq i \leq m}  x_i \geq \prod_{j \ne i} x_j);   \]
\[  t(u): \quad \forall w \ w \leq u; \]
\[  b(u): \quad \forall w \ w \geq u. \]
The last one is more complicated.    For a finite set $X$ and bounded intervals $I_1, \dots, I_n$, 
write $\op{CI}(X,I_1, \dots, I_n)$ to mean $X \subseteq \bigcup_{1 \leq j \leq n} I_j$.  The inclusion can be written as
$\&_{x \in X} \op{OR}_{1 \le j \le n} (x \in I_j )$.  
This is a positive first-order condition as we require each $I_j$ to be bounded, so e.g. $x \in I = [c,d]$ gets written out as $c \leq x \ \&\  x \leq d$, and so on.

\smallskip
\noindent Moreover, for a set $\{ x,y,z \}$ we define the following intervals:
\begin{align*}
I^x &= [x+yz, x+(x+y)(x+z)(y+z)] \\  %or just [x+yz, (x+y)(x+z)]
J_x &= [ x(xy+xz+yz), x(y+z)] \\          %or just xy+xz, x(y+z)]
K    &= [xy+xz+yz,(x+y)(x+z)(y+z)].
\end{align*}

\begin{theorem} The following sentence $\pi_3$ holds in $\mbf F_3$ but not in $\mbf F_4$:
\begin{align*}                 
& \exists z_1 \exists z_2 \exists z_3 \ t(z_1+z_2+z_3) \ \&\ b(z_1z_2z_3) \ \& \\
&  \forall x_1 \forall x_2 \forall x_3 \forall x_4 \ [\op{NI}(x_1, x_2, x_3, x_4) \  \op{OR} \\  
& \op{OR}_{i \ne j} \op{CI}(\{ x_1, \dots, x_4 \}, I^{z_i}, J_{z_j},K) \  \op{OR} \\ 
& \op{OR}_{i \le 3} \op{CI}(\{ x_1, \dots, x_4 \}, [z_i,z_i],K) ],
\end{align*}
where $ I^{z_i}, J_{z_j},K$ are with respect to the set $\{ z_1, z_2, z_3 \}$.
\end{theorem}

\begin{proof}
First we show that the sentence holds in $\mbf F_3$.  
Let $\mbf F_3$ be generated by $x,y,z$ and take $z_1=x$, $z_2=y$, $z_3=z$.
Let $x_1, \dots, x_4$ be elements of $\mbf F_3$.
In the presence of $\mathrm{(W)}$,  $\op{NI}(x_1, x_2, x_3, x_4)$ says exactly that those elements
do not generate a copy of $\mbf F_4$ by \cite[Corollary~1.12]{the_book}.
On the other hand,
Whitman showed that $\mbf F_3$ contains $\mbf F_\omega$ \cite{whitman}, see
\cite[Theorems~1.28 and~9.10]{the_book}.
%see Theorems~1.28 and~9.10 in \cite{the_book}. 
In particular, $\mbf F_3$ contains many copies of $\mbf F_4$, and the sentence $\pi_3$ restricts their location.
%But of course $\mbf F_3$ does contain copies of $\mbf F_4$, and we want to locate them.  
%By Tschantz's Theorem (\cite[Theorem~9.10]{the_book}), every infinite interval in $\mbf F_n$ contains a copy of $\mbf F_\omega$.

The free distributive lattice $\op{FD}_3$ is a bounded lattice.
Also, Alan Day's doubling construction preserves the property of being bounded \cite{RAD77, RAD79},\cite [Sec.~II.3]{the_book}.
%\cite{RAD77, RAD79, the_book}[Sec.~II.3].  maybe hack this better
By doubling the elements in $\op{FD}_3$ that are the join of two atoms,
or the meet of two coatoms, we obtain the lattice $\mbf A$ in Figure~\ref{fig:doubled}.
Thus the natural homomorphism $h : \mbf F_3 \to \mbf A$ is bounded.
The algorithm for computing lower and upper bounds 
for congruence classes of the kernel 
of a bounded homomorphism, which goes back to J\'onsson \cite{Jonsson} and 
McKenzie \cite{McKenzie}, is given in Theorem 2.3 of \cite{the_book}.
Applying this to the homomorphism $h$ decomposes $\mbf F_3$ into a disjoint union 
the congruence classes of $\ker h$.  These turn out to be
intervals of the following forms (up to permutations of variables):
\begin{align*}
& [u,u]               &(T^u)    \\ 
     &\quad \text{ for } u = (x+y)(x+z), x+y, x+z, x+y+z \\
&[x+yz, x+M]    &(I^x)\\
&[x,x]                 &(G_x) \\
& [xm, x(y+z)]   &(J_x) \\
& [v,v]               &(B_v) \\  
    & \quad \text{ for } v =  xy+xz, xy, xz, xyz\\
&[m,M]                &(K).
\end{align*}
where $m=xy+xz+yz$ and $M=(x+y)(x+z)(y+z)$.
Except for the singleton classes $T^u$ with $u \in \{ x+y, x+z,y+z,x+y+z \}$, and $B_v$ with $v \in \{ xy,xz,yz,xyz \}$,
this decomposition is sketched schematically in Figure~\ref{fig:intervals}.
%This partitioning of $\mbf F_3$ was motivated by Figures~3.4 and~3.5 of~\cite{the_book}, 
%which sketch the structure of the bottoms of $\mbf F_3$ and $\mbf F_4$, respectively.

\begin{figure}
\begin{center}
\begin{tikzpicture}[scale=1]
\tikzstyle{every node}=[draw,circle,fill=white,minimum size=5pt,inner sep=0pt,label distance=1mm] 
   \node (0) at (0,0)[]{};
   \node (1) at (-1,1)[]{};
   \node (2) at (0,1)[]{};
   \node (3) at (1,1)[]{};
   \node (4) at (-1,2)[]{};
   \node (5) at (0,2)[]{};
   \node (6) at (1,2)[]{};
   \node (7) at (-1,3)[]{};
   \node (8) at (0,3)[]{};
   \node (9) at (1,3)[]{};
   \node (10) at (0,4)[]{};
   \node (11) at (-1,1+4)[]{};
   \node (12) at (0,1+4)[]{};
   \node (13) at (1,1+4)[]{};
   \node (14) at (-1,2+4)[]{};
   \node (15) at (0,2+4)[]{};
   \node (16) at (1,2+4)[]{};
   \node (17) at (-1,3+4)[]{};
   \node (18) at (0,3+4)[]{};
   \node (19) at (1,3+4)[]{};
   \node (20) at (0,4+4)[]{};
   \node (21) at (-2,4)[label=left:$x$]{};
   \node (22) at (1,4)[label=right:$z$]{};
   \node (23) at (2,4)[label=right:$y$]{};

   \draw (10)--(9)--(6)--(3)--(0)--(1)--(4)--(7)--(10)--(8)--(5);
   \draw (0)--(2)--(4);
   \draw (2)--(6);
   \draw (1)--(5)--(3);
   \draw (10)--(11)--(14)--(17)--(20)--(19)--(16)--(13)--(10)--(12)--(15);
   \draw (20)--(18)--(16);
   \draw (18)--(14);
   \draw (19)--(15)--(17);
   \draw (7)--(21)--(11);
   \draw (8)--(22)--(12);
   \draw (9)--(23)--(13);
\end{tikzpicture}
\end{center}
\caption
{Lattice $\mbf A$ obtained by doubling six elements in $\op{FD}_3$} \label{fig:doubled} 
\end{figure}
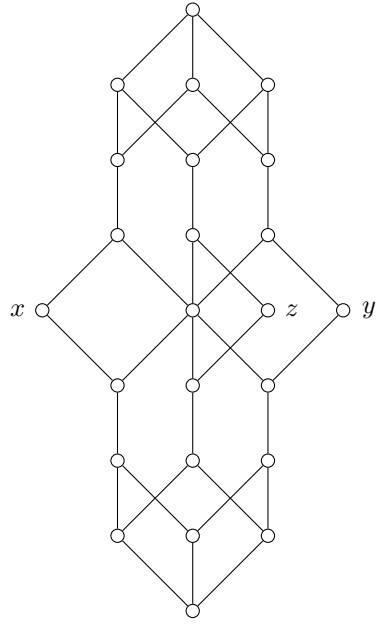

%Decompose $\mbf F_3$ into a union of intervals of the following forms (up to permutations of variables):
%\begin{align*}
%& [(x+y)(x+z), x+y+z]  &(T^x)\\
%&[x+yz, x+M]    &(I^x)\\
%&[x,x]                 &(G_x) \\
%& [xm, x(y+z)]   &(J_x) \\
%&[xyz, xy+xz]    &(B_x) \\
%&[m,M]                &(K).
%\end{align*}
%where $m=xy+xz+yz$ and $M=(x+y)(x+z)(y+z)$.
%These intervals are pairwise disjoint, except that the top three $T^x$, $T^y$ and
%$T^z$ overlap, and dually the bottom three $B_x$, $B_y$ and $B_z$ overlap.
%See Figure~\ref{fig:intervals}.
%This partitioning of $\mbf F_3$ was motivated by Figures~3.4 and~3.5 of~\cite{the_book}, 
%which sketch the structure of the bottoms of $\mbf F_3$ and $\mbf F_4$, respectively.

\begin{figure}
\begin{center}
\begin{tikzpicture}[scale=1.0]
   \node at (0,3.5) [draw=white,rectangle,fill=white] {$K$};
   \node at (-1.5,5.5) [draw=white,rectangle,fill=white] {$I^x$};
   \node at (-1.5,1.5) [draw=white,rectangle,fill=white] {$J_x$};
   \node at (1.5,5.5) [draw=white,rectangle,fill=white] {$I^y$};
   \node at (1.5,1.5) [draw=white,rectangle,fill=white] {$J_y$};
\tikzstyle{every node}=[draw,circle,fill=white,minimum size=5pt,inner sep=0pt,label distance=1mm] 
    \node (1) at (-2,0) [label=left:$xy+xz$] {};
    \node (2) at (-1,1) [] {};
    \node (3) at (0,2) [label=below:$m$] {};
    \node (4) at (1,1) [] {};
    \node (5) at (2,0) [label=right:$xy+yz$] {};
    \node (6) at (0,5) [label=above:$M$] {};

    \draw (3) .. controls (-1.5,3) and (-1.5,4) .. (6); 
    \draw (3) .. controls (1.5,3) and (1.5,4) .. (6); 
    
    \node (7) at (-2,7) [label=left:$(x+y)(x+z)$] {};
    \node (8) at (-1,6) [] {};
    \node (9) at (1,6) [] {};
    \node (10) at (2,7) [label=right:$(x+y)(y+z)$] {};
    \node (11) at (-1,4) [] {};
    \node (12) at (-2,5) [label=left:$x+yz$] {};

    \draw (12) .. controls (-1.4,5.1)  .. (8); 
    \draw (12) .. controls  (-1.7,5.9) .. (8); 
    
    \node (13) at (-3,3.5) [label=left:$x$] {};
    \node (14) at (-2,2) [label=left:$x(y+z)$] {};

    \draw (2) .. controls (-1.45,1.9)  .. (14); 
    \draw (2) .. controls  (-1.7,1.1) .. (14); 

   \node (15) at (-1,3) [] {};
   \node (16) at (1,4) [] {};
   \node (17) at (2,5) [] {};

    \draw (17) .. controls (1.35,5.1)  .. (9); 
    \draw (17) .. controls  (1.7,5.9) .. (9); 
    
   \node (18) at (3,3.5) [label=right:$y$] {};
   \node (19) at (2,2) [label=right:$y(x+z)$] {};

    \draw (4) .. controls (1.35,1.9)  .. (19); 
    \draw (4) .. controls  (1.7,1.1) .. (19); 

   \node (20) at (1,3) [] {};

    \draw (1)--(2)--(3)--(4)--(5);
    \draw (7)--(8)--(6)--(9)--(10);
    \draw (11)--(12)--(13)--(14)--(15);
    \draw (16)--(17)--(18)--(19)--(20);
 %   \draw (2)--(5); 
 %   \draw (4)--(7)--(6); 
%  \draw (5) .. controls (2.5,4) .. (6); 
 % \draw (5) .. controls (4.5,4) .. (6);   
\end{tikzpicture}

\caption
{Schematic of interval decomposition of $\mbf F_3$.}
\label{fig:intervals}
\end{center} 
\end{figure}
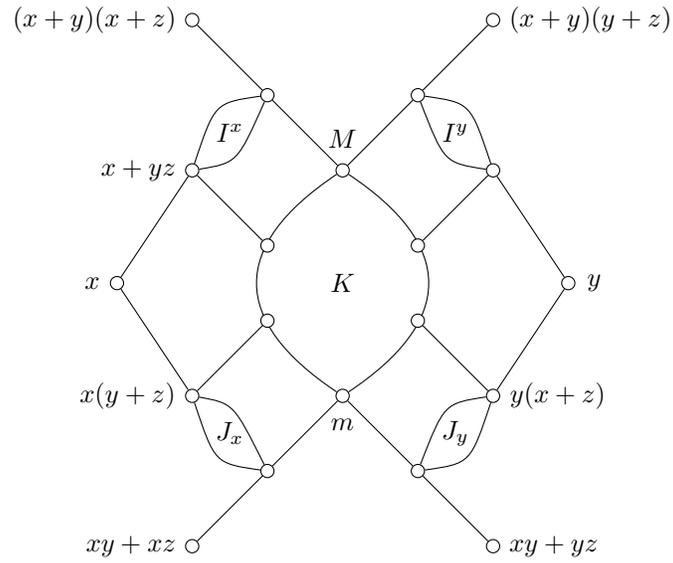

\smallskip
\noindent The claim is that if $x_1, \dots, x_4$ generate a copy of $\mbf F_4$, then
$x_1, \dots, x_4$ is contained in a union of intervals of the form 
$I^x \cup J_y \cup K$ or $\{ x \} \cup K$, up to permutations of variables.
Certainly we cannot have $x_i \in T^u$, that is $x_i=u$, or $x_i \in B_v$ for the elements at the top or bottom of $\mbf F_3$, as each such $\uparrow\! u$ and $\downarrow\! v$ contains only finitely many elements, whereas $\uparrow\! x_i$ and $\downarrow\! x_i$ are infinite.
%That is, none of the $x_i$ are in the sets $T^x$ (consisting of 1, coatoms and meets of 2 coatoms),
%$B_x$ (with 0, atoms and joins of 2 atoms), and 
Moreover, $\{ x_1, \dots, x_4 \}$ cannot contain elements from
%% or "intersect nontrivially "
both sets of any of the following pairs of intervals:
\begin{align*}
&\{ x \} \ \&\   I^x   \qquad  \{ x \} \ \&\  J_x  \qquad  I^x \ \&\   J_x \\
&\{ y \} \ \&\   I^x   \qquad  \{ y \} \ \&\  J_x  \qquad \{ y \} \ \&\ \{ x \}\\
&I^y \ \&\   I^x   \quad \qquad  J_y \ \&\  J_x 
\end{align*}
The first line is because it would make some pair $x_i$, $x_j$ comparable.
For the second line, note if $t \in I^x$ or $t=x$, then $y+ t = x+y$, and 
similarly for the third line, if $u \in I^x$ and $v \in I^y$, then $u+v=x+y$.
But $x+y$ is a coatom of $\mbf F_3$, and you cannot have $x_i + x_j$ being a coatom, 
no matter where the remaining $x_k$'s lie, since the filter $\uparrow\!(x_i+x_j)$ has at least 4 elements. 
Thus you cannot have the first entries in the 2nd and 3rd lines inhabited,
and dually for the second entries,
and for the same reason $\{ x_i,x_j \} \ne \{ y, x \}$.
That leaves options contained in unions of the form $I^x \cup J_y \cup K$ or $\{ x \} \cup K$, as claimed.

\smallskip
\noindent
Now to show that $\pi_3$ fails in $F_4$. 
The logical form of $\pi_3$ is  the following:
\[  \exists \mbf z \forall \mbf x \ A \ \&\  (B  \text{ or } C) . \]
Its negation would be the following: 
\[  \forall \mbf z \exists \mbf x \   A \rightarrow (\neg B  \ \&\  \neg C) . \]
We will show, switching the quantifiers, that the following holds:
\[  \exists \mbf x \forall \mbf z  \   A \rightarrow (\neg B  \ \&\  \neg C),  \]
which is slightly stronger, since it means that $\mbf x$ is chosen uniformly.

\smallskip
\noindent
We take $x_1, \dots, x_4$ to be the standard generators of $\mbf F_4$, for which $\op{NI}(x_1, \dots, x_4)$ fails,
and intend to show that there do not exist $z_1, z_2, z_3$ with $z_1+z_2+z_3=1$ and $z_1z_2z_3=0$
satisfying - up to symmetry - one of the following two conditions:
$$\op{CI}(\{ x_1, \dots, x_4 \}, I^{z_1}, J_{z_2},K) \;\text{ or }\;
\op{CI}(\{ x_1, \dots, x_4 \}, [z_1,z_1],K).$$

\noindent \underline{Case 1}. $\op{CI}(\{ x_1, \dots, x_4 \}, I^{z_1}, J_{z_2},K)$ holds.
\newline 
Then the least element of $J_{z_2}$, which is $z_2(z_1z_2 + z_1z_3+z_2z_3)$, is $0=x_1x_2x_3x_4$.
\emph{A fortiori} $z_1z_2+z_2z_3=0$, whence $z_1z_2 = 0 = z_2z_3$.
By $(\op{SD}_\meet)$, $z_2(z_1+z_3)=0$.   But that is the top element of $J_{z_2}$,
so there are no $x_i$'s in that interval.  Dually, there are no $x_j$'s in $I^{z_1}$.  
Thus they are all in $K = [z_1z_2 + z_1z_3 + z_2z_3, (z_1+z_2)(z_1+z_3)(z_2+z_3)]$.
But then $z_1z_2 + z_1z_3 + z_2z_3=0$.  Again all 3 joinands are $0$, and applying 
$(\op{SD}_\meet)$ twice (in its more general form, $u=ab=cd$ implies $u=(a+c)(a+d)(b+c)(b+d)$, 
cf.~\cite[Theorem~1.21]{the_book})
we get $(z_1+z_2)(z_1+z_3)(z_2+z_3)=0$.  However, $(z_1+z_2)(z_1+z_3)(z_2+z_3)=1$ since 
it is the top of $K$ and all the $x_j$'s are in $K$. Therefore, that is a contradiction. %This concludes Case 1.

\noindent \underline{Case 2}. $\op{CI}(\{ x_1, \dots, x_4 \}, [z_1,z_1],K)$ holds.
\newline  W.l.o.g. $z_1=x_1$ and $x_2,x_3,x_4 \in K$, 
else we revert to the previous case where everything is in $K$.
If, say, $z_2 \leq z_3$, then the bottom of $K$ is $z_1z_2 + z_1z_3 + z_2z_3 = z_2+z_1z_3$.
Thus $z_2$ and $x_1z_3$ are both below $x_2x_3x_4$, and since each $x_j$ is meet prime in $\mbf F_4$,
we get $z_3 \leq x_2x_3x_4$.  That contradicts $z_1+z_2+z_3= 1 = x_1 + \dots +x_4$.
Hence the $z_j$'s are incomparable and distinct.
On the other hand, from $z_1z_2z_3=0$ we get $\{ x_2,x_3,x_4 \} \gg \{ z_2,z_3 \}$, and dually
from $z_1+z_2+z_3=1$ we get $\{ x_2,x_3,x_4 \} \ll \{ z_2,z_3 \}$.
($A \gg B$ if $\forall a \exists b \ a \geq b$; $C \ll D$ is dual; these are not symmetric.)
Thus each of $x_2$, $x_3$, $x_4$ is above some $z_j$ and below some $z_k$, with $\{ j,k \} \subseteq \{ 2,3 \}$.
Assume say $x_2 \geq z_2$.
By $\{ x_2,x_3,x_4 \} \ll \{ z_2,z_3 \}$ either $x_2 \leq z_2$ or $x_2 \leq z_3$.
The latter gives $z_3 \leq x_2 \leq z_2$, contradicting the argument above.
Thus $x_2 = z_2$.  Similarly $x_3=z_3$.   
That makes $x_4 \geq z_2$ and $x_4 \geq z_3$ both impossible. 

\smallskip
\noindent
Hence, both cases lead to a contradiction, and we are done.
\end{proof}

\section{The profinite-bounded completion of $\mbf F_n$}

	\begin{definition}\label{doubly_prime} Let ${\mbf L}$ be a lattice, we say that $a \in {\mbf L}$ is doubly prime if it is both join prime and meet prime, that is for every $b, c \in {\mbf L}$, the following hold:
	\begin{enumerate}[(1)]
	\item $a \leq b \vee c$ implies $a \leq b$ or $a \leq c$;
	\item $b \wedge c \leq a$ implies $b \leq a$ or $c \leq a$.
	\end{enumerate}
\end{definition}

	\begin{fact}\label{doubly_prime_fact} Let $a$ be an element in a free lattice ${\mbf F}(X)$. The following are equivalent:
	\begin{enumerate}[(1)]
	\item $a \in X$,
	\item $a \in {\mbf F}(X)$ is doubly prime (cf.~Definition~\ref{doubly_prime}). 
\end{enumerate}	
\end{fact}

    \begin{definition} We say that the a model $A$ is prime if it embeds elementarily in every model of its first-order theory. We say that $A$ is minimal if it has no proper elementary substructures.
    \end{definition}

        \begin{fact}[{\cite[Proposition~5.1]{nies}}] \label{four}
        Let $A$ be a countable structure. Then $A$ is a prime model of its theory iff, for every $0 < n < \omega$, each orbit under the natural action of $\mrm{Aut}(A)$ on $A^n$ is first-order definable without parameters in $A$.
        \end{fact}

	\begin{lemma} Let $\kappa$ be a cardinal.
	\begin{enumerate}[(1)]
	\item If $\kappa$ is finite, then $\mbf F_\kappa$ is a prime and minimal model of its theory.
	\item If $\kappa = \omega$, then $\mbf F_\kappa$ is a prime model of its theory but it is not minimal.
	\item If $\kappa > \omega$, then $\mbf F_{\omega}$ embeds elementarily in $\mbf F_\kappa$.
	\end{enumerate}
\end{lemma}

	\begin{proof} We first prove that if $\kappa \leq \aleph_0$, then $\mbf F_\kappa = {\mbf F}(X)$ (so $|X| = \kappa$) is a prime model of its theory. To this extent, by Fact~\ref{four} it suffices to show that for every $n < \omega$ and for every $n$-tuple $\bar{a}$ of elements of $\mbf F_\kappa$, the $\mrm{Aut}(\mbf F_\kappa)$-orbit of $\bar{a} = (a_1, ..., a_n)$ is first-order definable in $\mbf F_\kappa$ without paramenters (notice that under our assumptions $\mbf F_\kappa$ is countable). For every $1 \leq i \leq n$, let $t_i(\bar{x})$ be a term in the variables $\bar{x} \subseteq X$ (adding variables possibly not actually occurring in $t_i(\bar{x})$ we can assume that $\bar{x}$ is the same for all the $i$'s) such that $t_i^{{\mbf F_\kappa}}(\bar x) = a_i$. Then $\bar{b} = (b_1, ..., b_n)$ is in the $\mrm{Aut}(\mbf F_\kappa)$-orbit of $\bar{a}$ iff there is a tuple $\bar{y}$ of the same length as $\bar{x}$ of doubly prime elements of $\mbf F_\kappa$ such that for every $1 \leq i \leq n$ we have $t_i^{{\mbf F_\kappa}}(\bar y) = b_i$, and by Fact~\ref{doubly_prime_fact} this is first-order. This proves that $\mbf F_\kappa$ is prime for $\kappa \leq \aleph_0$. The claim about the minimality and not minimality of $\mbf F_\kappa = {\mbf F}(X)$ is clear, as if $|X| = \aleph_0$ is infinite, then the sublattice generated by an infinite proper subset of $X$ is also a prime model of its theory, while if $X$ is finite this is not the case, as the existence of exactly $n < \omega$ doubly prime elements is expressible in first order logic. Finally, concerning (3), it suffices to show that for every finite subset $\{a_1, ..., a_m \}$ of $\mbf F_{\omega}$ and every element $b \in \mbf F_{\kappa}$ there exists an automorphism of $\mbf F_{\kappa}$ which fixes $\{a_1, ..., a_m \}$ and maps $b$ into $\mbf F_{\omega}$, and this is easy to see (this is a well-known general fact).
\end{proof}

	We now introduce the crucial notions of lower and upper bounded lattices. Our treatment of the subject will be brief, for more see e.g. \cite[Chapter~II]{the_book} or \cite{nation}.

	\begin{definition}\label{bounded_def} Let $\mbf K$ and $\mbf L$ be lattices. A homomorphism $f: \mbf K \rightarrow \mbf L$ is said to be lower bounded if for every $a \in \mbf L$, the set $\{ u \in \mbf K: f(u) \geq a \}$ is either empty or has a least element. A finitely generated
lattice $\mbf L$ is called lower bounded if every homomorphism $f: \mbf K \rightarrow \mbf L$, where $\mbf K$ is finitely generated, is lower bounded. Let $D_0(\mbf L)$ denote the set of join prime elements of\/ $\mbf L$, i.e., those
elements which have no nontrivial join-cover (cf.~\cite[pg. 29]{the_book}). For $k > 0$, let $a \in
D_k (\mbf L)$ if every
nontrivial join-cover $V$ of $a$ has a refinement $U \subseteq D_{k-1}(\mbf L)$ which is also a join-cover of $a$. Then a finitely generated lattice $\mbf L$ is lower bounded if and only if\/ $\bigcup_{k < \omega} D_{k}(\mbf L) = L$.
%
%%jn
Observe that, from the definition, $D_0(\mbf L) \subseteq D_1(\mbf L) \subseteq D_2(\mbf L) \subseteq \cdots$. Thus, if\/ $\mbf L$ is
lower bounded and $a \in \mbf L$, we define $\rho(a)$, the $D$-rank of $a$, to be the least integer $k$ such that $a \in D_k(\mbf L)$.  
Every finitely generated lower bounded lattice has the minimal join-cover refinement property \cite[Cor.~2.19]{the_book}, which implies that every element is a finite join of join irreducible elements.  Thus we can define the $D$-rank of a lower bounded lattice $L$ to be $\mrm{sup}\{\rho(a) : a \in \op{J}(\mbf L )\}$. 
The notions of upper bounded homomorphism, upper bounded lattice, etc.~are defined dually. In the upper case we write $D^{\mrm{op}}$-rank of $a$, to distinguish the two notions.
\end{definition}

	\begin{fact}[\cite{nation}]\label{bounded_fact} Let $k < \omega$. The class of lattices all of whose finitely generated sublattices are lower and upper bounded of $D$-rank and $D^{\mrm{op}}$-rank $\leq k$ (cf.~Definition~\ref{bounded_def}) forms a variety $\mathcal{V}_k$. We denote by $\mbf B_{(n, k)}$ the free object of rank $n$ in the variety $\mathcal{V}_k$ and with $h_k$ the canonical homomorphism of\/ $\mbf F_n$ onto  $\mbf B_{(n, k)}$. Also, for $k < \ell < \omega$, the variety $\mathcal{V}_k$ is contained in the variety $\mathcal{V}_\ell$, and so, for fixed $n < \omega$, we let $f_{(k, \ell)}$ be the canonical homomorphism from $\mbf B_{(n, \ell)}$ onto $\mbf B_{(n, k)}$. Further, $(\mbf B_{(n, k)}, f_{(k, \ell)} : \ell \leq k < \omega)$ is an inverse system. 

 \smallskip
 \noindent
 Finally, we denote by $\mbf H_n$ the inverse limit of the inverse system $(\mbf B_{(n, k)}, f_{(k, \ell)} : \ell \leq k < \omega)$. The lattice $\mbf F_n$ embeds canonically into $\mbf H_n$. We refer to $\mbf H_n$ as the profinite-bounded completion of\/ $\mbf F_n$.
\end{fact}

	\begin{lemma}\label{homomor_lemma} Let $\mbf K$ be an elementary extension of\/ $\mbf F_n = \mbf F(X)$ with $X = \{x_1, ..., x_n \}$ and let $\mbf L$ be a finite, upper and lower bounded lattice generated by $\{a_1, ..., a_n \}$. Then the homomorphism $f: \mbf F_n \rightarrow \mbf L$ such that $x_i \mapsto a_i$ extends canonically to a homomorphism $\hat{f}^{\mbf K} = \hat{f} : \mbf K \rightarrow \mbf L$. Further, in the context of Fact~\ref{bounded_fact}, $(\mbf K, \hat{h}^{\mbf K}_k : k < \omega)$ is a cone of the inverse system $(\mbf B_{(n, k)}, f_{(k, \ell)} : \ell \leq k < \omega)$, and thus there is a homomorphism $h: \mbf K \rightarrow \mbf H_n$ which commutes with $(\mbf B_{(n, k)}, f_{(k, \ell)} : \ell \leq k < \omega)$.
\end{lemma}

	\begin{proof} As $L$ is bounded, necessarily $f: \mbf F_n \rightarrow \mbf L$ is bounded, i.e., the kernel of $f$ partitions $\mbf F_n$ into bounded congruence classes, that is, every congruence class $f^{-1}(a)/\mrm{ker}(f)$, for $a \in \mbf L$, has a least element $\beta(a)$ and a greatest element $\alpha(a)$. Thus, the equivalence classes of this partition are intervals of the form $[\beta(a), \alpha(a)]$ for $a \in \mbf L$, and for $u \in \mbf F_n$ we have that $f(u)=a$ iff $\beta(a) \leq u \leq \alpha(a)$. 
 Again, the algorithm for computing lower and upper bounds 
$\beta$ and $\alpha$ for congruence classes of the kernel 
of a bounded homomorphism
%, which goes back to J\'onsson \cite{Jonsson} and McKenzie \cite{McKenzie}, 
is given in Theorem 2.3 of \cite{the_book}.
% The algorithm for computing lower and upper bounds 
%$\beta$ and $\alpha$ for congruence classes of the kernel 
%of a bounded homomorphism, which goes back to J\'onsson \cite{Jonsson} and 
%McKenzie \cite{McKenzie}, is given in Theorem 2.3 of \cite{the_book}.

 As $\mbf K$ is an elementary extension of $\mbf F_n$, this also determines a partition of $\mbf K$, let $\psi_{\mbf K}$ be the corresponding equivalence relation on $\mbf K$. Define then $\hat{f}^{\mbf K} = \hat{f} : \mbf K \rightarrow \mbf L$ as $\hat{f}(u) = a$ iff $u \in a/\psi_{\mbf K}$. We claim that $\hat{f}$ is a homomorphism. We show that $\hat{f}$ preserves joins, a dual argument works for meets. To this extent, remember that $\beta(c) \join \beta(d) = \beta(c \join d)$,
and $\alpha (c) \join \alpha(d) \leq \alpha(c \join d)$.
Thus if $u, v \in \mbf K$ and
$\beta(c) \leq u \leq \alpha(c)$ and $\beta(d) \leq v \leq \alpha(d)$,
then $\beta(c \join d) \leq u \join v \leq \alpha(c \join d)$. Finally, the ``further part'' of the lemma is easy as $\mbf H_n$ is the inverse limit of $(\mbf B_{(n, k)}, f_{(k, \ell)} : \ell \leq k < \omega)$.
\end{proof}

	\begin{example} In Figure~\ref{fig:pentagon} we see how $\mbf F_3$ (and thus any elementary extension $\mbf K$ of\/ $\mbf F_3$) gets partitioned into the pentagon $\mbf N_5$ via the natural map of\/ $\mbf F_3$ onto $\mbf N_5$.
    Note that the pentagon is in the variety $\mathcal V_1$, since $D_1(\mbf N_5) = \mbf N_5$.   
 
 %where we recall that $\mbf N_5 \cong \mbf B_{(3, 0)}$ (recall Fact~\ref{bounded_fact} and the notation introduced there).
	
 \begin{figure}\label{pentagon}%[H]
\begin{center}
\begin{tikzpicture}[scale=.9]
\tikzstyle{every node}=[draw,circle,fill=white,minimum size=4pt,inner sep=0pt,label distance=1mm] 
    \node (0) at (0,0) [label=right:$xyz$] {};
    \node (1) at (0,2) [label=right:$z(x+y)$] {};
    \node (2) at (1,3) [label=right:$z$] {};
    \node (3) at (0,4) [label=right:$z+xy$] {};
    \node (4) at (0,6) [label=right:$x+y+z$] {};
    \node (5) at (-1,5) [label=left:$x \join y$] {};
    \node (6) at (-1,3) [label=right:$w$] {};
    \node (7) at (-2.35,3.65) [label=left:$y+z(x+y)$] {};
    \node (8) at (-2.35,1.65) [label=left:$xy$] {};

    \draw (0)--(8)--(6)--(3)--(2)--(1)--(7)--(5)--(4); 
    \draw (0) ..controls (-.55,1) .. (1);
    \draw (0) ..controls (.55,1) .. (1);
    \draw (3) ..controls (-.55,5) .. (4);
    \draw (3) ..controls (.55,5) .. (4);
    \draw (6) ..controls (-1.55,4) .. (5);
    \draw (6) ..controls (-.45,4) .. (5);
    \draw (7) ..controls (-2.9,2.65) .. (8);
    \draw (7) ..controls (-1.8,2.65) .. (8);
\end{tikzpicture}
\end{center} 
\caption{The bounded congruence classes of the natural map of $\mbf F_3$ onto the pentagon $\mbf N_5$, where we let $w=x(z+xy)$.}
\label{fig:pentagon}
\end{figure}
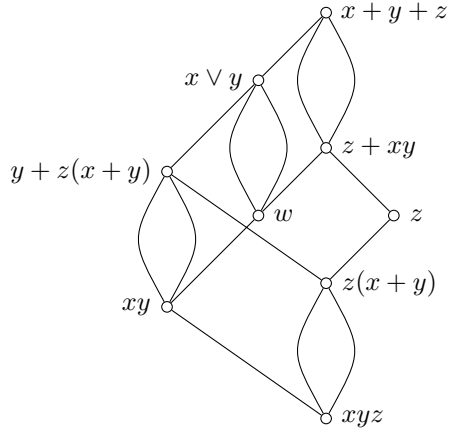
\end{example}

	\begin{notation}\label{psi_notation} As in the proof of \ref{homomor_lemma}, letting $\mbf K$ be an elementary extension of\/ $\mbf F_n$ we denote by $\psi = \psi_{\mbf K}$ the congruence induced by the homomorphism $h: \mbf K \rightarrow \mbf H_n$.
\end{notation}

	\begin{proposition}\label{delta_prop} In the context of Lemma~\ref{homomor_lemma} and Notation~\ref{psi_notation}:
	\begin{enumerate}[(1)]
	\item $\psi = \bigcap_{k < \omega} \varphi_k$, where $\varphi_k$ is the kernel of the homomorphism $\hat{h}^{\mathbf{K}}_k: \mbf K \rightarrow \mbf B_{(n, k)}$;
	\item $\psi \restriction_{\mbf F_n} = \Delta$ (the identity).
\end{enumerate}	
\end{proposition}

	\begin{proof} (1) is clear, (2) is because of \cite[Theorem~4.1]{the_book}, which is Day's theorem that free lattices are weakly atomic
    \cite{RAD77}; see also \cite[Section~II.7]{the_book}.
\end{proof}

%%jn
Let us elaborate on Day's theorem, which plays a crucial role in the arguments below. 
For each $w \in \mbf F_n$ there is a congruence $\varphi_w$ which is maximal with respect to the
property that $(v,w) \notin \varphi$ for any $v < w$.  It turns out that $\mbf F_n/\varphi_w$ is a finite lower bounded lattice of $D$-rank $\rho(w)$.  Moreover, a join irreducible element $w$ is completely join irreducible if and only if $\varphi_w$ is both lower and upper bounded. 
In that case, $\mbf F_n/\varphi_w$ is in $\mathcal V_n$ where $n$ is the complexity of $w$,
that is, $w \in X^{\meet(\join\meet)^n}$.

Day's theorem says that the completely join irreducible elements of $\mbf F_n$ are join-dense,
i.e., if $u \nleq v$ in $\mbf F_n$ then there exists a completely join irreducible element
$w$ with $w \leq u$ and $w \nleq v$.  Thus every element in a free lattice is the join of the
completely join irreducible elements below it.

	\begin{fact}\label{observation_beta_n} As in the proof of \ref{homomor_lemma}, letting $h_k: {\mbf F_n} \rightarrow {\mbf B_{(n, k)}}$ be the canonical homomorphism and letting $\alpha_k(u)$ and $\beta_k(u)$ be, respectively, the greatest and least element of the equivalence class $u/\mrm{ker}(h_k)$, then to every element $u \in {\mbf F_n}$ we can associate sequences $(\beta_k(u) : k < \omega), (\alpha_k(u) : k < \omega) \in ({\mbf F_n})^{\omega}$ such that:
\begin{equation}\tag{$\star$}
\beta_0(u) \leq \beta_1(u) \leq \cdots \leq u \leq \cdots \leq \alpha_1(u) \leq \alpha_0(u).
\end{equation}
In fact, to every element $\mbf{c} = (c_k : k < \omega) \in \mbf{H}_n$, letting $\alpha_k(\mbf{c})$ be the greatest element of $h^{-1}_{k}(c_k) \subseteq \mbf{F}_n$ and $\beta_k(\mbf{c})$ be the least element of $h^{-1}_{k}(c_k) \subseteq \mbf{F}_n$ we have:
\begin{equation}\tag{$\star\star$}
\beta_0(\mbf{c}) \leq \beta_1(\mbf{c}) \leq \cdots \leq \cdots \leq \alpha_1(\mbf{c}) \leq \alpha_0(\mbf{c}).
\end{equation}
For $\mbf{a} \in \mbf{F}_n$, this double sequence corresponds to the sequence $(\star)$ identifying $\mbf{F}_n$ with its canonical embedding into $\mbf{H}_n$ (cf.~what has been said at the end of \ref{bounded_fact}). As a piece of notation, given $\mbf{c} \in H$ we also write $\mbf{c} = (\mbf{b}, \mbf{a})$, where we let:
$$\mbf{b} = (\beta_j(\mbf{c}) = b_j : j < \omega) \;\; \text{ and } \;\; \mbf{a} = (\alpha_j(\mbf{c}) = a_j : j < \omega).$$
With respect to this identification, the lattice order of\/ $\mbf{H}_n$ can be characterized as:
$$\mbf{c} \leq \mbf{c}' \;\; \Leftrightarrow \;\; \mbf{b} \leq \mbf{b}' \;\; \Leftrightarrow \;\; b_j \leq b'_j \; \text{ for all } j < \omega$$
$$\mbf{c} \leq \mbf{c}' \;\; \Leftrightarrow \;\; \mbf{a} \leq \mbf{a}' \;\; \Leftrightarrow \;\; a_j \leq a'_j \; \text{ for all } j < \omega.$$
\end{fact}

An element $w$ in a lattice $\mbf L$ is \emph{lower atomic} if 
for all $u < w$ there exists $v$ such that $u \leq v \prec w$.
The dual condition is called \emph{upper atomic}, and $w$ is
\emph{totally atomic} if it is both lower and upper atomic.  

For an element $w$ of a finitely generated free lattice $\mbf F_n$,
there is a one-to-one correspondence between the lower
covers of $w$ and its completely join irreducible canonical joinands.
If every canonical joinand is completely join irreducible, then $w$ is
lower atomic.  The dual statements hold for upper atomic and completely
meet irreducible canonical meetands.  See Theorem~3.5 and Corollary~3.8
of \cite{the_book}, expanded in Theorem~3.26 and Corollary~3.27.

\begin{fact}\label{four_type_elements} In the context of \ref{observation_beta_n}, there are four types of elements in $\mbf F_n$:
\begin{enumerate}[(1)]
	\item  the ones such that the $\beta$'s and the $\alpha$'s are eventually constant;
	\item the ones such that the $\beta$'s are eventually constant but the $\alpha$'s are not;
	\item the ones such that the $\alpha$'s are eventually constant but the $\beta$'s are not;
	\item the ones such that neither the $\beta$'s nor the $\alpha$'s are eventually constant.
\end{enumerate}
Furthermore, the four types above admit the following algebraic description:
\begin{enumerate}[(1$'$)]
	\item $c \in \mbf F_n$ is as in (1) iff $c$ is totally atomic;
	\item $c \in \mbf F_n$ is as in (2) iff $c$ is lower atomic but not upper atomic;
	\item $c \in \mbf F_n$ is as in (3) iff $c$ is upper atomic but not lower atomic;
	\item $c \in \mbf F_n$ is as in (4) iff $c$ is neither upper nor lower atomic.
\end{enumerate} 
Notice that by \cite[Chapter~6]{the_book} the number of elements of type (1$'$) is finite.
%\todog{We do not need this but this fact seems nice to state. I am sure that this is true but I am not sure what would be the best reference for it from the blue book. Could you help?}
\end{fact}

%	\begin{definition}\label{join_irredu} Let $\mbf L$ be a lattice. We say that $a \in \mbf L$ is completely join irreducible (resp. completely meet irreducible) if it has a unique lower (resp. upper) cover.
%\end{definition}

	Concerning the Dedekind-MacNeille completion (below) see e.g. \cite[pp. 165-169]{priestley}.

	\begin{notation} Let $\mbf P$ be a poset. We denote by $\op{DM}(\mbf P)$ the Dedekind-MacNeille completion of\/ $\mbf P$, i.e., the set of subsets of $P$ such that $C = C^{u \ell}$, where for $D \subseteq P$:
	$$D^u = \{ p \in P : \forall d \in D, \; d \leq p\} \; \text{ and } \; D^\ell = \{ p \in P : \forall d \in D, \; p \leq d\}.$$
$\op{DM}(\mbf P)$ is ordered by inclusion. 
Recall that $\op{CJI}(\mbf P)$ (resp. $\op{CMI}(\mbf P)$) denotes the set of completely join irreducible (resp. completely meet irreducible) elements of\/ $\mbf P$. For $k < \omega$, we let $\op{CJI}_k(\mbf P)$ denote the set of completely join irreducible elements of\/ $\mbf P$ of $D$-rank $\leq k$. We shorten completely join irreducible with $\op{CJI}$ and, when clear from the context, we write $\op{CJI}_k$ and $\op{CJI}$ \mbox{instead of\/ $\op{CJI}_k(\mbf P)$ and $\op{CJI}(\mbf P)$.}
\end{notation}
%Let $\op{CJ}(\mbf F_n)$ denote the set of completely join irreducible (cf. Definition~\ref{join_irredu}) elements of $\mbf F_n$
%and $\op{CJ}_k$ the completely join irreducible elements of rank $\leq k$.  
%Note for the record that $\op{CJ}(\mbf F_n)$ is $\op D$-closed (\cite{the_book}), but not meet-closed (\cite{the_book}).

\begin{theorem}  \label{lemma:claim2} %\label{theorem:HDM}
For $\mbf F_n$ with $n$ finite, $\mbf H_n \cong \op{DM}(\mbf F_n)$.
\end{theorem}

\begin{proof}
For a sequence $\mbf c = ( \mbf b,\mbf a )$ in $\mbf H_n$, let us set notation:
\begin{align*}
A &= \{ a_j : j \in \omega \} \\
B &= \{ b_j : j \in \omega \} \\
A^\ell &= \bigcap \downarrow\! a_j \\
\downarrow\! B &= \bigcup \downarrow\! b_j
\end{align*}

\begin{lemma} \label{lemma:inter}
$\downarrow\! B \,\cap\, \op{CJI} = A^\ell \,\cap\, \op{CJI}$.  
\end{lemma}

\begin{proof}
First we show $\downarrow\! B \,\cap\, \op{CJI} \subseteq A^\ell$.
Consider $b_j$ and arbitrary $k$, w.l.o.g.~$k \geq j$ as the $a_j$'s are descending.  
Then $b_j \leq b_k \leq a_k$.
Next $A^\ell \,\cap\, \op{CJI} \subseteq\  \downarrow\! B$.  
Let $w \in \mrm{LHS}$ of rank $j$.  Then $w \leq a_j$ implies $w = h_j(w) \leq h_j(a_j)$,
whence $w = \beta h_j(w) \leq \beta h_j(a_j) = b_j$.
\end{proof}

\begin{lemma} \label{lemma:inter2}
Let\/ $I$ and $J$ be $\op{DM}$-closed ideals of\/ $\mbf F_n$.
If\/ $I \,\cap\, \op{CJI} = J \,\cap\, \op{CJI}$, then $I = J$. 
\end{lemma}

\begin{proof}
$\op{DM}$-closed ideals, which are of the form $D^{\ell}$, are closed under joins
that exist in $\mbf F_n$.  On the other hand, by Day's theorem, every element of $\mbf F_n$
is a (possibly infinite) join of CJI elements.
\end{proof}

\smallskip
\noindent Now we set about establishing the isomorphism of Theorem~\ref{lemma:claim2}.
Define a map
$\varphi : \mbf H_n \to \op{DM}(\mbf H_n)$ 
by $\varphi(\mbf c) = A^\ell$
for each $\mbf c = ( \mbf b,\mbf a )$ in $\mbf H_n$.
Recall that $\mbf{c} \leq \mbf{c}' \;\; \Leftrightarrow \;\; \mbf{b} \leq \mbf{b}'$ by Fact~\ref{observation_beta_n}.
Then:
\begin{enumerate}[(1)]
\item  $B \subseteq A^\ell$ and $A^\ell \in \op{DM}(\mbf F_n)$.
\item  $\varphi$ is order-preserving:  $\mbf b \leq \mbf b'$ means $b_j \leq b_j'$ for all $j$. 
          It follows that $a_j \leq a_j'$ for all $j$, so $A^\ell \subseteq (A')^\ell$.
\item  $\varphi$ is 1-to-1:  $\mbf b \nleq \mbf b'$ implies there exists $j$ with $b_j \nleq b_j'$.
         Then by Day's theorem \cite{RAD77} there is a CJI $u$ such that $u \leq b_j$ and $u \nleq b_j'$, whence
         $h_j(b_j') \ngeq u$.  Thus $a_j' \ngeq u$, as in general $h_j(u) \leq v$ iff $u \leq \alpha (v)$,  
         whence $a_j' \ngeq b_j$.
\item $\varphi$ is onto:   let $I = I^{u\ell}$ be a $\op{DM}$-closed ideal.  
         Define a sequence $b_j = \Join (I \cap \op{CJI}_j)$, which makes sense because $\op{CJI}_j$ is finite.
         Because every element of $\mbf F_n$ is a join of CJI elements and 
        $B = \{ b_j : j \in \omega \} \subseteq I$, we have $I \cap \op{CJI} = \  \downarrow\! B \,\cap\, \op{CJI}$.
        By Lemma~\ref{lemma:inter} this yields $I \,\cap\, \op{CJI}= A^\ell \,\cap\, \op{CJI}$, whence by 
        Lemma~\ref{lemma:inter2}, $I = A^\ell = \varphi(\mbf c)$ is in the range of $\varphi$.         
\end{enumerate}
\end{proof}

\begin{lemma} \label{lemma:claim1} 
Let $\mbf L_0$ be a lattice admitting a fixed-point free polynomial $p(x, \mbf {b})$, with $\mbf {b}$ a finite sequence in $\mbf L_0$ and let $\mbf L_1$ be a complete lattice containing $\mbf L_0$.
Then there is a positive $\forall\exists$-sentence true in $\mbf L_1$ and false in $\mbf L_0$.
\end{lemma}

\begin{proof}
For $j = 0, 1$, let $p_j: \mbf L_j \to \mbf L_j$ be such that $a \mapsto p(a, \mbf {b})$ (i.e., the corresponding polynomial function).   Then, recalling that $\mbf L_1$ is a complete lattice, and recalling also the choice of $p(x, \mbf {b})$, the following two things happen:
$$\mbf L_1 \models \forall \mbf {b} \exists x (p(x, \mbf {b}) = x) \;\;\; \text{ and } \;\;\; \mbf L_0 \not\models \forall \mbf {b} \exists x (p(x, \mbf {b}) = x).$$
%%jbn
Indeed, let $A_1 = \{ a \in L_1 : a \leq p(a, \mbf b) \}$ and $a_1 = \Join A_1$.
It is straightforward that:
$$p(a_1,\mbf b) = a_1.$$
This last equality is the Tarski fixed-point theorem; see \cite{Knaster, Tarski} 
and the discussion in Section~I.5 of \cite{the_book}.
\end{proof}

%Lemma \ref{lemma:claim2} was here.

%Claim 2. The inverse limit $H$ of the finite bounded lattices $L_k$'s from page 8, Section 6 of 6th Installment is isomorphic to $\op{DM}(\mbf F_n)$ 
%= the Dedekind McNeill completion of $\mbf F_n$ = the usual completion of a lattice by normal ideals or by cuts.
%Proof. Let $f_k: H \to  L_k$ canonically. We define $f: \op{DM}(F_n) \to H$ by $A \mapsto (\bigvee f_k[A] : k < \omega) \in H$, for $A \in \op{DM}(\mbf F_n)$.

%jbn removed comments
\begin{theorem} \label{62no}
Let $3 \leq n < \omega$. Then $\mbf H_n$ is not elementarily equivalent to $\mbf F_n$, in fact there is a positive $\forall\exists$-sentence true in $\mbf H_n$ and false in $\mbf F_n$.
\end{theorem}

\begin{proof} This follows immediately from Lemmas \ref{lemma:claim2} and \ref{lemma:claim1} and the existence of fixed-point free polynomials in $\mbf F_n$, see e.g. \cite[Section~I.5.]{the_book}.
\end{proof}

\section{Proof of Theorem \ref{the_last_theorem}}

\newcommand{\up}{u}  %{\uparrow}  
\newcommand{\dn}{\ell} %{\downarrow}  

%Let $\mbf L$ be a lattice.
%For $D \subseteq L$, we set
%\begin{align*}
%D^\up &= \{ w \in L : w \geq d \text{ for all } d \in D \}, \\
%D^\dn &= \{ w \in L : w \leq d \text{ for all } d \in D \}, \\
%\kappa(D) &= D^{\up\dn} 
%\end{align*}
%Theorem~\ref{lemma:claim2} is that for $3 \leq n < \omega$, $\mbf H_n \cong %\op{DM}(\mbf F_n)$. %%{lemma:claim2}{4.14}
%Theorem \ref{62no}  is that for $3 \leq n < \omega$, $\mbf H_n \not\equiv \mbf F_n$. %%{4.18}

The ideals of a lattice $\mbf L$, ordered by set inclusion, form a complete lattice $\op{Id}(\mbf L)$.  
Likewise the filters, ordered by set inclusion, form a complete lattice $\op{Fil}(\mbf L)$.
(If $\mbf L$ has a least element $0$, we normally take $\{ 0 \}$ to be the least ideal;
if not then the empty set is contained in $\op{Id}(\mbf L)$.  Dually for $\op{Fil}(\mbf L)$.
The filter lattice is often ordered by \emph{reverse set inclusion}; for current purposes,
set inclusion is more natural.  That makes the join of two filters $F \join G$ to be the filter
generated by $F \cup G$, directly analogous to the join of two ideals in $\op{Id}(\mbf L)$. Recall that 
$\op{DM}(\mbf F_n) \cong \mbf H_n$ by Theorem~\ref{lemma:claim2}, 
so in what follows we identify the two objects.

%Let $\kappa : \op{Id}(\mbf L) \to \op{DM}(\mbf L)$ via $\kappa(I) = I^{\up\dn}$.
%\bigskip
%\textcolor{red}{The new part:}

\begin{theorem} \label{theorem:bh}
Let $\kappa : \op{Id}(\mbf F_n) \to \op{DM}(\mbf F_n)$ via $\kappa(I) = I^{\up\dn}$. Then the map $\kappa$ witnesses that $\mbf H_n$ is a retract of the ideal lattice $\op{Id}(\mbf F_n)$.
\end{theorem}

Our path to Theorem~\ref{theorem:bh} includes a more general result, i.e., our Theorem~\ref{the_last_theorem}. To this extent, let us start by recalling some basic facts.

\begin{lemma} \label{lemma:std}
Let\/ $\mbf L$ be a lattice. Then:
\begin{enumerate}[(1)]
\item  For any $D \subseteq L$, $D^\up \in \op{Fil}(\mbf L)$.
\item An ideal $I$ of\/ $\mbf L$ is $\op{DM}$-closed, i.e., $\kappa(I)=I$, if and only if\/
$I = F^\dn$ for some filter $F$ of\/ $\mbf L$, where $\kappa : \op{Id}(\mbf L) \to \op{DM}(\mbf L)$ is defined as $\kappa(I) = I^{\up\dn}$.
\end{enumerate}
\end{lemma}

\begin{lemma} \label{lemma:calc}
Let\/ $\mbf K$ be a lattice satisfying the following assumptions:
\begin{enumerate}[(1)]
   \item $\mbf K$ satisfies $\mathrm{(W)}$;
   \item $\mbf K$ is generated by a set $X$ of join prime elements.
\end{enumerate}
Let $F$, $G$ be filters of \/ $\mbf K$.  Then:
\begin{enumerate}[(i)]
\item  $F^\dn \join G^\dn = (F \cap G)^\dn$;
\item $F^\dn \cap G^\dn = (F \join G)^\dn$.  %% i.e., (\op{Fg}(F \cup G))^\dn
\end{enumerate}
Hence $\op{DM}(\mbf K)$ is a sublattice of\/ $\op{Id}(\mbf K)$.
\end{lemma}

\begin{proof}
Part (ii) is general nonsense.  The other direction being trivial,
we need that $F^\dn \cap G^\dn \subseteq (F \join G)^\dn$.
If $w \in F^\dn \cap G^\dn$, then $w \leq f$ for all $f \in F$,
and $w \leq g$ for all $g \in G$.  Therefore $w \leq f \meet g$
for any $f \in F$, $g \in G$, and these are the generators for
the filter generated by $F \cup G$. 

\smallskip \noindent For part (i) we need to show the nontrivial direction 
$F^\dn \join G^\dn \supseteq (F \cap G)^\dn$.
Following a standard method of proof and using the contrapositive inclusion, let:
\[ \mbf S = \{ w \in K : \text{ for all } F,G \in \op{Fil}(\mbf K) : w \notin F^\dn \join G^\dn \text{ implies }w \notin (F \cap G)^\dn  \} .  \]
We will prove that $\mbf S$ is a sublattice of $\mbf K$ containing the generating set $X$,
which means $\mbf S = \mbf K$ and the statement is true.

\smallskip \noindent If $w \in X$ and $w \notin F^\dn \join G^\dn$, then in particular $w \notin F^\dn \cup G^\dn$.
Then there exist $f_1 \in F$ and $g_1 \in G$ such that $w \nleq f_1$ and $w \nleq g_1$.
As $w$ is join prime, $w \nleq f_1 \join g_1$ which is in $F \cap G$.
Hence $w \notin (F \cap G)^\dn$.  Therefore $X \subseteq S$.

\smallskip \noindent If $w = w_1 \join w_2$ with $w_1$, $w_2 \in S$ and
$w \notin  F^\dn \join G^\dn$, then one join, say $w_1$, is such that
$w_1 \notin F^\dn \join G^\dn$.   Since $w_1 \in S$ we have $w_1 \notin (F \cap G)^\dn$, which is an ideal.
Hence $w_1 \join w_2 \notin (F \cap G)^\dn$, and so $\mbf S$ is closed under joins.

\smallskip \noindent Now assume $w = w_1 \meet w_2$ with $w_1$, $w_2 \in S$ and $w \notin F^\dn \join G^\dn$.   
Then $w_1 \notin  F^\dn \join G^\dn$ so as $w_1 \in S$ we get $w_1 \notin (F \cap G)^\dn$, 
whence there exists $h_1 \in F \cap G$ with $w_1 \nleq h_1$.   Likewise there exists $h_2 \in F \cap G$
such that $w_2 \nleq h_2$.   
Also $w \notin F^\dn$ whence $w \nleq f_1$ for some $f_1 \in F$.
Similarly $w \nleq g_1$ for some $g_1 \in G$.

\smallskip \noindent Now we claim that by $\mathrm{(W)}$, we have
\[  w_1 \meet w_2 = w \nleq (f_1 \meet h_1 \meet h_2) \join (g_1 \meet h_1 \meet h_2) \in F \cap G .\]
Note that $f_1 \meet h_1 \meet h_2 \in F$ and $g_1 \meet h_1 \meet h_2 \in G$, so their join is in $F \cap G$.
Thus $w \notin (F \cap G)^\dn$, as desired, and so $\mbf S$ is closed under meets.

\smallskip \noindent In view of Lemma~\ref{lemma:std}, (i) and (ii) show that the join and meet
of two $\op{DM}$-closed ideals is another $\op{DM}$-closed ideal.
Thus $\op{DM}(\mbf K)$ is a sublattice of\/ $\op{Id}(\mbf K)$.
\end{proof}

A classic result of K. Baker and A. Hales \cite{baker} deals with the connection
between a lattice and its ideal lattice:

\begin{lemma} \label{corollary:bhp} Let $\mbf L$ be a lattice.  Then
\begin{enumerate}[(1)]
\item  $\op{Id}(\mbf L) \in \mbb{HSU}(\mbf L)$;
\item  $\op{Id}(\mbf L)$ and $\mbf L$ share the same universal positive theory;
\item  if\/ $\mbf L$ satisfies $\mathrm{(W)}$, then $\op{Id}(\mbf L)$ satisfies $\mathrm{(W)}$.
\end{enumerate}
\end{lemma}

%Combining this with Lemma~\ref{lemma:calc}, we get:

%\begin{corollary}\label{corollary:bhpp}
%If\/ $\mbf K$ satisfies the conditions of\/ \emph{Lemma~\ref{lemma:calc}}, 
%then $\op{DM}(\mbf K)$ satisfies $\mathrm{(W)}$.
%Moreover, $\op{DM}(\mbf K)$ satisfies every positive, universal, first-order sentence holding in $\mbf K$.
%\end{corollary}

%Taking $\mbf K = \mbf F_n$ with $3 \leq n < \omega$, we see that 
%$\mbf H_n$ satisfies Whitman's condition $\mrm{(W)}$.
%The second part does not say much, since the positive universal sentences of $\mbf F_n$
%are true in all lattices.

Lemma~\ref{lemma:calc} and \ref{corollary:bhp} already give the first part of our Theorem~\ref{the_last_theorem}. Thus, to establish the rest of Theorem~\ref{the_last_theorem}, we are only left to show Theorem~\ref{theorem:bh}, i.e., that $\kappa: \op{Id}(\mbf F_n) \to \op{DM}(\mbf F_n)$ is a homomorphism,
which obviously fixes $\op{DM}$-closed sets.  For this we need Lemma~\ref{lemma:calc} and the corresponding statements for 
$I^\up$ and $J^\up$ where $I$ and $J$ are ideals, with $u$ and $\ell$ interchanged.
Both versions apply to free lattices because the elements in the generating set are both join and meet prime.
To this extent, we calculate:
%\todog{Using the "dual" is slightly confusing, as one might thing that has to put $\vee$ instead of $\wedge$ and $\cup$ instead of $\cap$, where you actually mean to leave these as they are but replace $\ell$ with $u$, right? E.g., in the first application of (i) dual below. If I got things right, can you add a clarifying sentence?}
\begin{align*}
\kappa(I \join J) &= (I \join J)^{\up\dn} \qquad  \text{by definition}\\
                           &= (I^\up \cap J^\up)^\dn \qquad \text{by (ii) dual} \\
                           &= I^{\up\dn} \join J^{\up\dn} \qquad\    \text{by (i)} \\
                           &= \kappa(I) \join \kappa(J)\quad\,  \text{by definition,}
\end{align*}
\begin{align*}
\kappa(I \meet J) &= (I \cap J)^{\up\dn} \qquad\quad  \text{by definition}\\
                           &= (I^\up \join J^\up)^\dn \quad \text{by (i) dual} \\
                           &= I^{\up\dn} \cap J^{\up\dn} \qquad\quad    \text{by (ii)} \\
                           &= \kappa(I) \meet \kappa(J)\qquad\!  \text{by definition.}
\end{align*}

    We denote by $\op{SD}_\join$ semidistributivity with respect to $\vee$ and similarly for $\wedge$.

\begin{corollary} \label{cor:props}
  The ideal lattice $\op{Id}(\mbf F_n)$ and the profinite-bounded completion $\op{DM}(\mbf F_n) \cong \mbf H_n$ have the following properties:
  \begin{enumerate}[(1)]
      \item $\op{Id}(\mbf F_n)$ satisfies $\mathrm{(W)}$ and\/ $\op{SD}_\join$, but fails $\op{SD}_\meet$ for $n \geq 3$;
      \item $\mbf H_n$ satisfies $\mathrm{(W)}$ and both semidistributive laws.  
  \end{enumerate}
\end{corollary}

\begin{proof}
    Whitman's condition for $\op{Id}(\mbf F_n)$ follows from Lemma~\ref{corollary:bhp}.
    Join semidistributivity of $\op{Id}(\mbf F_n)$ is proved by F. Wehrung in Corollary~5.4 of \cite{Wehrung}; see the comment immediately after the corollary. 
    By Theorem ~\ref{theorem:bh} these properties of $\op{Id}(\mbf F_n)$ are inherited by its retract $\mbf H_n$.  
    Moreover, the construction of $\mbf H_n$ is self-dual, so it also satisfies $\op{SD}_\meet$.

    It remains to show that $\op{Id}(\mbf F_n)$ fails meet semidistributivity.  In $\mbf F_3$, let
    \begin{align*}
        y_0 &= y     &\qquad   z_0 &= z  \\
        y_{k+1} &= y + xz_k  &\qquad  z_{k+1} &= z + xy_k 
    \end{align*}
    In $\op{Id}(\mbf F_3)$, let $X = \downarrow\! x$, 
    $Y = \bigcup_{k \geq 0} \downarrow\! y_k$, and
    $Z = \bigcup_{k \geq 0} \downarrow\! z_k$.
    Then $X \meet Y = X \meet Z < X \meet (Y \join Z)$,
    since $x(y+z)$ is in the latter but not the first two.    
\end{proof}

\end{document}